\documentclass[twoside,12pt,reqno]{amsart}

\date{\today}

\usepackage{accents}
\usepackage[latin1]{inputenc}
\usepackage{pstricks}
\usepackage{enumerate}	
\usepackage{graphicx,color}
\usepackage{caption}
\usepackage[english]{babel}
\usepackage[toc,page]{appendix}
\usepackage{amssymb, mathrsfs, amsfonts, amsmath,amsthm}
\usepackage{amsbsy, hyperref}
\usepackage{latexsym}
\usepackage{booktabs}
\usepackage{tikz}
\usepackage{eurosym} 
\usepackage{newfloat}
\usepackage{hyperref}
\usepackage[alphabetic, msc-links]{amsrefs}

\usepackage{mathtools}

\definecolor{r}{rgb}{.9,0.1,.3}

\numberwithin{equation}{section}

\newtheorem{Proposition}{Proposition}[section]
\newtheorem{Definition}{Definition}[section]
\newtheorem{Theorem}{Theorem}[section]
\newtheorem{Lemma}{Lemma}[section]

\newtheorem{Remark}{Remark}[section]
\newtheorem{claim}{Claim}[section]

\newcommand{\R}{\mathbb{R}}

\newcommand{\ee}{\mathbf e}
\newcommand{\p}{\partial}

\newcommand{\dist}{\operatorname{dist}}
\newcommand{\Div}{\operatorname{div}}

\newcommand{\spt}{\operatorname{spt}}

\author[E. S. Gama]{Eddygledson S. Gama}
\address[Gama]{
  Departamento de Matem\'atica,
  Universidade Federal do Cear\'a, Bloco 914, Campus do Pici,
  Fortaleza, Cear\'a, 60455-760, Brazil.
}
\email{eddygledson@gmail.com}

\thanks{
E. S. Gama is supported by Coordena\c{c}\~ao de
Aperfei\c{c}oamento de Pessoal de N\'{i}vel Superior - Brasil CAPES/PDSE/88881.132464/2016-01.}

\title[Translating solitons in $\R^{n+1}$]{Translating solitons $C^1-$asymptotic to two half-hyperplanes}

\begin{document}
\maketitle
\begin{abstract} 
We prove that the hyperplanes parallel to ${\bf e}_{n+1}$ are the unique examples of translating solitons $C^1-$asymptotic to two half-hyperplanes outside a vertical cylinder in $\R^{n+1}$. This result generalizes previous result due to F. Mart\'in and the author in \cite{Gama-Martin}. 
\end{abstract}

\section{Introduction}
An oriented hypersurface $M$ in $\R^{n+1}$ is called a translating soliton (or translator) if 
$$M+ t \; \mathbf{e}_{n+1}$$
is a mean curvature flow. This is equivalent to 
\[
\vec{\bf{H}}={\bf e}_{n+1}^{\bot},
\]
where $\vec{\bf{H}}$ denotes the mean curvature vector field of $M$ and $\bot$ indicates the projection over the normal bundle of $M.$ Thus we have the scalar mean curvature satisfies:
\begin{equation}\label{TS-Eq.}
H=\langle N,{\bf e}_{n+1}\rangle,
\end{equation}
where $N$ indicates the unit normal along of $M.$ Recall that $H$ is just the trace of the second fundamental form of $M.$

In 1994, T. Ilmanen \cite{Ilmanen} showed that translating solitons are minimal hypersurfaces in $\R^{n+1}$ endowed with the conformal metric $g:=e^{\frac{2}{n}x_{n+1}}\langle\cdot.\cdot\rangle.$ From now on, we shall always assume that $\R^{n+1}$ is endowed with the metric $g.$ 

We will say that a translating soliton $M$ in $\R^{n+1}$ is {\bf complete} if $M$ is complete as hypersurface in $\R^{n+1}$ with the Euclidean metric.

This duality of being able to see translating solitons as minimal hypersurfaces was the key point that allowed to F. Mart\'in and the author  \cite{Gama-Martin} the use of  tools from the theory of varifolds to concluded that translating solitons $C^1$-asymptotic to two half-hyperplanes outside no vertical cylinder in $\R^{n+1}$ must be either a hyperplane parallel to $\ee_{n+1}$ or an element of the family of the tilted grim reaper cylinder. 

The family of the tilted grim reaper cylinders is the family of graphs given by the one-parameter family of functions \[f_\theta:\left(-\frac{\pi}{2\cos \theta},\frac{\pi}{2\cos \theta}\right)\times\R^{n-1}\to\R\] given by $f_\theta(x_1\ldots,x_n)=x_n\tan(\theta)-\sec^2(\theta)\log\cos(x_1\cos(\theta)), \; \theta\in[0,\pi/2)$. Besides the previous result, the method used in \cite{Gama-Martin} also implies, for a vertical cylinder and dimension $n < 7$, that the hypersurface $M$ must coincide with an hyperplane parallel to $\ee_{n+1}$. Thus it remained open to know whether the same type of result were true for any dimension, i. e., if the hyperplane parallel to $\ee_{n+1}$ are the unique examples that outside a vertical cylinder are $C^1-$asymptotic to two half-hyperplanes. 

In this paper, we will prove that variation of the method used in \cite{Gama-Martin}, together with the result about connectness of the regular set of a stationary varifold due to Ilmamen in \cite{Ilmanen-maximum} and a sharp version of the maximum principle due to N. Wickramasekera in \cite{Wickramasekera} allow us to conclude that the hyperplanes parallel to $\ee_{n+1}$ are the unique examples of translating solitons $C^1-$asymptotic to two half-hyperplanes outside a vertical cylinder in $\R^{n+1}$ for all dimension.

It is important to point out here that the main theorem  of this paper and the main theorem obtained in \cite{MPGSHS15} (for dimension three) and in \cite{Gama-Martin} (for arbitrary dimension) give a complete characterization of all the translators which are $C^1-$asymptotic to two half-hyperplanes outside a cylinder in $\R^{n+1}$, up to rotations fixing $\ee_{n+1}$ and translations. More precisely, the next result holds for all dimensions. Here ${\bf u}_\theta=-\sin (\theta) \cdot  \ee_n+\cos(\theta) \cdot \ee_{n+1}.$
\begin{Theorem} \label{th:41}
Let $M\hookrightarrow \mathbb{R}^{n+1}$ be a complete, connected, properly embedded translating soliton and consider the cylinder $\textstyle{\mathcal{C}_\theta(r):=\{x\in \mathbb{R}^{n+1} \; : \; \langle x,\ee_{1}\rangle^2+\langle {\bf u}_{\theta},x\rangle^2 \leq r^2\}},$ where $r>0.$ Assume that $M$ is $C^{1}$-asymptotic to two half-hyperplanes outside $\mathcal{C}_\theta(r)$.
\begin{itemize}
\item[i.] If $\theta\in[0,\pi/2)$, then we have one, and only one, of these two possibilities:
\begin{enumerate}
\item[a.] Both half-hyperplanes are contained in the same hyperplane $\Pi$ parallel to $\ee_{n+1}$ and $M$ coincides with $\Pi$;
\item[b.] The half-hyperplanes are included in different parallel hyperplanes and $M$ coincides with a vertical translation of the tilted grim reaper cylinder associated to $\theta$.
\end{enumerate}
\item[ii.]If $\theta=\pi/2$, then $M$ coincides with a hyperplane parallel to $\ee_{n+1}$.
\end{itemize}
\end{Theorem}

Notice that this theorem is sharp in several senses.  If we increase the number of half-hyperplane then there are a lot of counterexamples. The cylinder over the pitchfork translator obtained recently by D. Hoffman, F. Mart\'in and B. White in \cite{Hoffman-New}  is an example of a complete, connected, properly embedded translating soliton which is $C^1-$asymptotic to 4 half-hyperplanes outside a cylinder in $\R^{n+1}$. In general, the cylinder over the examples obtained by X. Nguyen in \cite{Nguyen-1},\cite{Nguyen-2} and \cite{Nguyen-3} give similar examples  which are $C^1-$asymptotic to $2 k$ half-hyperplanes outside a cylinder, for any $k \geq 2$. The examples given by Nguyen have infinity topology, however the pitchfork translator is simply connected. 

We would like to point that the number of asymptotic half-hyperplanes cannot be odd, because each loop in $\R^{n+1}$ must intersect each properly embedded hypersurface in $\R^{n+1}$ at a even number of points (counting their multiplicity), so whenever this example existed, we could find a loop so that it would intersect the example at an exactly odd number of points. 

On the other hand, the hypothesis about the asymptotic behaviour outside a cylinder is also necessary as it is shown by the examples  obtained by Hoffman, Ilmanen, Mart\'in and White in \cite{Hoffman}.

\section*{Acknowledgements}
I would like to thank Francisco Mart\'in for valuable conversations and suggestions about this work. I would like to thank the referee for his valuable suggestions about the manuscript.

\section{Preliminaries}\label{Preliminaries}
Let $\Pi$ be a hyperplane in $\R^{n+1}$ and $\nu$ an unit normal along $\Pi$ with respect to the Euclidean metric in $\R^{n+1}$. Suppose that $u:\Omega\to\R$ is a smooth function. The set ${\rm Graph}^\Pi[u]$ defined by
\[{\rm Graph}^{\Pi}[u]:=\{x+u(x)\nu\colon x\in\overline{\Omega}\},\] 
is called the graph of $u$. Notice that we can orient ${\rm Graph}^{\Pi}[u]$ by the unit normal
\[
N=\frac{1}{W}(\nu- D u),
\]
where $Du$ indicates the gradient of u on $\Pi$ with respect to the Euclidean metric and $W^2=1+|D u|^2.$

With this orientation for ${\rm Graph}^{\Pi}[u]$ we have that $x\mapsto\langle N,\nu\rangle$ is a positive Jacobi field on ${\rm Graph}^\Pi[u]$ of the Jacobi operator associated to the metric $g$ in $\R^{n+1}$. Therefore applying the similar strategy used by Shariyari in \cite{SHAHRIYARI15} we shall conclude that all the graphs are stable in $\R^{n+1}$ with the metric $g$. 

Actually, these graphs satisfy a stronger property than stability. Using the method developed by Solomon in \cite{Solomon} (see also \cite{Gama}) we shall conclude that graphs are area-minimizing inside the cylinder over their domain. More precisely, we have the next proposition. Here $\mathcal{A}_g[\Sigma]$ indicates the area of the hypersurface $\Sigma$ in $\R^{n+1}$ with the Ilmanen's metric $g.$
\begin{Proposition}\label{homology-general.}
Let $u:\overline{\Omega}\to\R$ a smooth function over a domain $\overline{\Omega}\subset\Pi$, where $\Pi$ is a hyperplane in $\R^{n+1}$. Suppose that ${\rm Graph}^{\Pi}[u]$ is a translating graph in $\R^{n+1}.$ Assume that $\Sigma$ is any another hypersurface inside the cylinder $\mathcal{C}(\Omega)=\{x+s\nu\colon x\in\overline{\Omega}\ {\rm and}\ s\in\R\}$ so that $\p \Sigma=\p {\rm Graph}^{\Pi}[u].$ Then we have
\[
\mathcal{A}_g[{\rm Graph}^{\Pi}[u]]\leq \mathcal{A}_g[\Sigma].
\]
\end{Proposition}
\begin{proof}
Let $N=\frac{1}{W}(\nu-D u)$ be the unit normal along ${\rm Graph}^{\Pi}[u].$ Suppose first that $\Sigma$ is a hypersurface in $\mathcal{C}(\Omega)$ that lies oneside of ${\rm Graph}^{\Pi}[u]$ and let $U$ be the domain in $\mathcal{C}(\Omega)$ limited by $\Sigma$ and ${\rm Graph}^{\Pi}[u]$. Consider the vector field $X$ in $\mathcal{C}(\Omega)$ obtained from the unit normal $N$ of ${\rm Graph}^{\Pi}[u]$ by parallel transport across the line of the flow of $\nu$. That is, $X$ is given by 
\[
X(p,s)=\frac{e^{x_{n+1}}}{W}(\nu-D u)\ {\rm for\ all}\ (p,s)\in\mathcal{C}(\Omega).
\]
Using that ${\rm Graph}^{\Pi}[u]$ satisfies \eqref{TS-Eq.} in $\R^{n+1}$ one gets
\begin{eqnarray*}
\Div_{\R^{n+1}} X=0.
\end{eqnarray*}
Thus the divergence theorem applying to $U$ and $X$ in $\R^{n+1}$ with to the Euclidean metric implies, up to a sign, that
\begin{eqnarray*}
0&=&\int_{{\rm Graph}^{\Pi}[u]}\langle X,N\rangle{\rm d}\mu_{{\rm Graph}^{\Pi}[u]}-\int_{\Sigma}\langle X,N_{\Sigma}\rangle{\rm d}\mu_\Sigma\\
&\geq&\int_{{\rm Graph}^{\Pi}[u]}e^{x_{n+1}}{\rm d}\mu_{{\rm Graph}^{\Pi}[u]}-\int_{\Sigma}e^{x_{n+1}}{\rm d}\mu_\Sigma=\mathcal{A}_g[{\rm Graph}^{\Pi}[u]]-\mathcal{A}_g[\Sigma].
\end{eqnarray*}
This completes the proof when $\Sigma$ lies oneside of ${\rm Graph}^{\Pi}[u]$. The general case can be obtained by breaking the hypersurface $\Sigma$ into many parts so that each part lies oneside of ${\rm Graph}^{\Pi}[u].$ 
\end{proof}
\begin{Remark}
This Proposition also was obtained by Xin in \cite{Xin} for $\nu=\ee_{n+1}.$
\end{Remark}

Next we define what means a hypersurface be $C^1-$asymptotic to a half-hyperplane.
\begin{Definition}\label{Def. Asymptotic}
Let $\mathcal{H}$ a open half-hyperplane in $\mathbb{R}^{n+1}$ and $w$ the unit inward pointing normal of $\partial \mathcal{H}$. For a fixed positive number $\delta$, denote by $\mathcal{H}(\delta)$ the set given by
\begin{equation*}
\mathcal{H}(\delta):=\left\{p+tw:p\in \partial \mathcal{H}\ \operatorname{and}\ t>\delta\right\}.
\end{equation*}
We say that a smooth hypersurface $M$ is $C^{k}-$asymptotic to the open half-hyperplane $\mathcal{H}$ if $M$ can be represented as the graph of a $C^{k}-$ function $\varphi:\mathcal{H}\longrightarrow \mathbb{R}$ such that for every $\epsilon>0$, there exists $\delta>0$, so that for any $j\in\{1,2,\ldots,k\}$ it holds
\begin{equation*}
\sup_{\mathcal{H}(\delta)}|\varphi|<\epsilon\ \operatorname{and}\ \sup_{\mathcal{H}(\delta)}|D^j\varphi|<\epsilon.
\end{equation*}
We say that a smooth hypersurface $M$ is $C^{k}-$asymptotic outside a cylinder to two half-hyperplanes $\mathcal{H}_{1}$ and $\mathcal{H}_{2}$ provided there exists a solid cylinder $\mathcal{C}$ such that:
\begin{itemize}
\item[i.]The solid cylinder $\mathcal{C}$ contains the boundaries of the half-hyperplane $\mathcal{H}_{1}$ and $\mathcal{H}_{2}$,
\item[ii.]$M\setminus \mathcal{C}$ consists of two connected components $M_{1}$ and $M_{2}$ that are  $C^{k}-$asymptotic to $\mathcal{H}_{1}$ and $\mathcal{H}_{2}$, respectively.
\end{itemize}
\end{Definition}

\begin{Remark}
Observe that the solid cylinders in $\R^{n+1}$ in the definition are those isometric to $D(r)\times\R^{n-1},$ where $D(r)$ indicates the disk of radius $r$ in $\R^2$.
\end{Remark}

We need some notation from the theory of varifolds (see \cite{Simon} for more information about this subject). Let $V$ be an n-dimensional varifold in $U$, where $U$ is an open subset of $\R^{n+1}.$
\begin{Definition}
We define ${\rm reg} V$ as the set of all the points $p\in U\cap\spt V$ so that there exist a open ball $B_\epsilon(p)\subset U$ such that $B_\epsilon(p)\cap\spt V$ is hypersurface of class $C^1$ in $B_\epsilon(p)$ without boundary. The set ${\rm reg} V$ is called the regular set of $V$. The complement of ${\rm reg} V$ in $U$, denoted by ${\rm sing} V:=U\setminus{\rm reg} V$, is called the singular set of $V.$
\end{Definition}

\begin{Definition}
We say that an $n-$dimensional varifold $V$ is connected provided that $\spt V$ is a connected subset in $U.$
\end{Definition}

The following compactness result (in the class varifolds) was proven in \cite{Gama-Martin}. 
\begin{Lemma}\label{remarkcompctness}
Let $M^n\hookrightarrow \mathbb{R}^{n+1}$ be a complete, connected, properly embedded translating soliton and $\displaystyle{\mathcal{C}(r):=\{x\in \mathbb{R}^{n+1} \; : \; \langle x,\ee_{1}\rangle^2+\langle x,  \ee_{n}\rangle^2 \leq r^2\}},$ for $r>0.$ Assume that $M$ is $C^{1}$-asymptotic to two half-hyperplanes outside $\mathcal{C}(r)$. Suppose that $\left\{b_{i}\right\}_{i\in \mathbb{N}}$ is a sequence in $[\ee_1,\ee_{n}]^\perp$ and let $\left\{M_{i}\right\}_{i\in \mathbb{N}}$ be a sequence of hypersurfaces given by $M_{i}:=M+b_{i}.$ Then there exist a connected stationary integral varifold $M_\infty$ and a subsequence $\{M_{i_k}\}\subset\{M_i\}$ so that
\begin{itemize}
\item[$($i$)$]$M_{i_k}\stackrel{*}{\rightharpoonup} M_{\infty}$ in $\R^{n+1}$;
\item[$($ii$)$]${\rm sing}\ M_\infty$ satisfies $\mathcal{H}^{n-7+\beta}({\rm sing}\ M_{\infty}\cap (\R^{n+1}\setminus \mathcal{C}(r)))=0$ for all $\beta>0$ if $n\geq7$, ${\rm sing}\ M_{\infty}\cap (\R^{n+1}\setminus \mathcal{C}(r))$  is discrete if $n=7$ and ${\rm sing}\ M_{\infty}\cap (\R^{n+1}\setminus \mathcal{C}(r))=\varnothing$ if $1\leq n\leq6$;
\item[$($iii$)$] $M_{i_k}\to\spt M_{\infty}$ in $\R^{n+1}\setminus (\mathcal{C}(r)\cup{\rm sing}\ M_{\infty}).$
\end{itemize}
\end{Lemma}

\section{Main theorem}\label{Main theorem}\label{Main}
Now we are going to see how we can get the main result from the results in Section \ref{Preliminaries}.
\begin{Theorem}\label{limitcase}
Let $M^n\hookrightarrow \mathbb{R}^{n+1}$ be a complete, connected, properly embedded translating soliton and $\displaystyle{\mathcal{C}(r):=\{x\in \mathbb{R}^{n+1} \; : \; \langle x,\ee_{1}\rangle^2+\langle x,  \ee_{n}\rangle^2 \leq r^2\}},$ for $r>0.$ Assume that $M$ is $C^{1}$-asymptotic to two half-hyperplanes outside $\mathcal{C}(r)$. Then $M$ must coincide with a hyperplane parallel to $\ee_{n+1}.$
\end{Theorem}
\begin{proof} 
We start by proving the following.
\begin{claim}\label{parallel}
The half-hyperplanes  $\mathcal{H}_{1}$ and $\mathcal{H}_{2}$ are parallel. 
\end{claim}
\begin{proof}[Proof of the Claim \ref{parallel}]
Assume this is not true, then we could take a hyperplane parallel to $\ee_{n+1}$, $\tilde{\Pi}$, such that it does not intersect $M$ and such that the normal vector $v$ to $\tilde{\Pi}$ is not perpendicular to $w_{1}$ and $w_{2}$, where $w_i$ denotes the unit inward pointing normal of $\p \mathcal{H}_i$. Next we translate $\tilde{\Pi}$ by $t_0 \in \R$ in the direction of $v$ until we get a hyperplane $\tilde{\Pi}_{t_0}:=\tilde{\Pi}+t_{0}v$ in such a way that either $\tilde{\Pi}_{t_0}$ and $M$ have a first point of contact or $\dist \left(\tilde{\Pi}_{t_0},M\right)=0$ and $\tilde{\Pi}_{t_0}\cap M=\varnothing.$ The first case is not possible by the maximum principle. The second case implies that there exists a sequence $\{p_i=(p_1^i,\ldots,p_{n+1}^i)\}\in M$ so that $\lim_i\dist(p_i,\tilde{\Pi})=0$, $\langle p_i,\ee_1\rangle\to a_1$ and $\langle p_i,\ee_n\rangle\to a_n$. In particular, we also have $(a_1,0,\ldots,0,a_n,0)\in\tilde{\Pi}$. Consider the sequence of hypersurface 
\(
\{M_i:=M-(0,p_2^i,\ldots,p_{n-2}^i,0,p_{n+1}^i)\}.
\)
By Lemma \ref{remarkcompctness}, after passing to a subsequence, $M_i\stackrel{*}{\rightharpoonup} M_\infty$, where  $M_\infty$ is a connected stationary integral varifold in $\R^{n+1}$ and $(a_1,0,\ldots,0,a_n,0)\in\spt M_\infty$ by \cite{Simon}[Corollary 17.8]. Thus by \cite{White}[Theorem 4] we conclude $\spt M_\infty=\tilde{\Pi}$, which is impossible because  of the behaviour of $\spt M_\infty$ outside $\mathcal{C}(r).$
\end{proof}

It is clear that neither $\mathcal{H}_{1}\subset\mathcal{H}_{2}$ nor $\mathcal{H}_{2}\subset\mathcal{H}_{1}.$

Denote by $\Pi_{1}$ and $\Pi_{2}$ the hyperplanes that contain the half-hyperplanes $\mathcal{H}_{1}$ and $\mathcal{H}_{2},$ respectively. Observe that the previous claim implies that $\Pi_{1}$ and $\Pi_{2}$ are parallel. We would like to conclude that $\Pi_1=\Pi_2.$ Assume that the contrary of this is true, i. e. admit that the hyperplanes $\Pi_1$ and $\Pi_2$ are different.

\begin{claim}\label{slablimitation}
$M$ lies in the slab between $\Pi_1$ and $\Pi_2.$
\end{claim}
\begin{proof}[Proof of the Claim \ref{slablimitation}]
Let $S$ be the closed slab limited by $\Pi_1$ and $\Pi_2$ in $\R^{n+1}.$ If  $M\setminus S\neq\varnothing,$ then proceeding as in the first paragraph we could find a hyperplane parallel $\tilde{\Pi}$ to $\Pi_1$ in $\R^{n+1}\setminus S$ so that either $\tilde{\Pi}$ and $M$ have a point of contact or $\dist(\tilde{\Pi},M)=0$. However, arguing as in the first paragraph, and taking in account the behaviour of $M$, we would conclude that both situations are impossible. So $M$ must lie in $S.$ Notice that $M$ does not intersect neither $\Pi_1$ nor $\Pi_2$, by the maximum principle.
\end{proof}

Next, we need to study the behaviour of $M$ inside the solid cylinder $\mathcal{C}(s)$ with respect to the hyperplane $\Pi_1$ and $\Pi_2.$ 
\begin{claim}\label{Cylinder bounded}
For all $s\geq r$ and $j\in\{1,2\}$ we have $\dist(M\cap\mathcal{C}(s),\Pi_j)>0$.
\end{claim}
\begin{proof}[Proof of the Claim \ref{Cylinder bounded}]
Otherwise we could find a sequence $\{p_i=(p_1^i,\ldots,p_{n+1}^i)\}$ in $M\cap\mathcal{C}(s)$ so that \(\dist(p_i,\Pi_j)=0\), so considering the sequence of hypersurfaces \(\{M_i:=M-(0,p_2^i,\ldots,p_{n-1}^i,0,p_{n+1}^i)\}\) by Lemma \ref{remarkcompctness} we would have that $M_i\stackrel{*}{\rightharpoonup} M_\infty$, after passing to a subsequence, where $M_\infty$ is a connected n-dimensional stationary integral varifold. Using that $\{p_i\}$ lies in $\mathcal{C}(s)$ we may also assume $\langle p_i,\ee_1\rangle\to a_1$ and $\langle p_i,\ee_n\rangle\to a_n$. Now $(a_1,0,\ldots,0,a_n,0)\in\spt\ M_\infty\cap \Pi_j$. So by \cite{White}[Theorem 4] we would have $\spt M_\infty=\Pi_j,$ which is impossible because $\Pi_1\neq\Pi_2$ and part of $\spt M_\infty$ is close to $\Pi_1$ and $\Pi_2.$ 
\end{proof}

We know,  thanks to our hypothesis over $M$, that $M\setminus \mathcal{C}(r)={\rm Graph}^{\Pi_1}[\varphi_1]\cup{\rm Graph}^{\Pi_2}[\varphi_2]$, where $\varphi_j:\mathcal{H}_j\to\R$ is a smooth function and it holds
\[
\sup_{\mathcal{H}_j(\delta)}|\varphi_j|<\epsilon\ {\rm and}\ \sup_{\mathcal{H}_j(\delta)}|D\varphi_j|<\epsilon,
\]
where $\delta$ depends on $\epsilon$ and $\delta\to+\infty$ as $\epsilon\to0.$ Fix some $s>r$ and define \[\epsilon:=\frac{1}{10}\min_{j}\{\dist(M\cap\mathcal{C}(s),\Pi_j)\}>0.\] For this choice of $\epsilon$ we take $\delta>0$ so that
\[
\sup_{\mathcal{H}_j(\delta)}|\varphi_j|<\epsilon\ {\rm and}\ \sup_{\mathcal{H}_j(\delta)}|D\varphi_j|<\epsilon.
\]

If we assume that $\Pi_1\neq\Pi_2,$ then these choices lead us to a contradiction as follows: let $\nu$ be the unit normal vector to $\Pi_{1}$ pointing outside $S$ and define $s_{0}=\dist(\Pi_1,\Pi_2)>0$. Notice that for this choice of $s_0$ we have that $M+s_{0}\nu$ does not intersect $S,$ but the wing of $M+s_0\nu$ corresponding to $\mathcal{H}_2+s_0\nu$ is asymptotic a half-hyperplane on $\Pi_1$ with the unit inward pointing normal to its boundary is $-w_1$.

Define $M_\epsilon:=\{x\in M\colon \min\{\dist(x,\Pi_1),\dist(x,\Pi_2)\}\geq\epsilon\}$. By Claim \ref{Cylinder bounded} one has $M\cap\mathcal{C}(s)\subset M_\epsilon.$ Notice that, taking $t_0>0$ sufficiently large $t_{0}>0$, we may assume that $M_\epsilon+s_{0}\nu+t_{0}w_{1}$  lies in $\mathcal{Z}_{1,2\delta}^{+}$ and $\mathcal{C}(s)\cap(\mathcal{C}(s)+s_{0}\nu+t_{0}w_{1})=\varnothing,$ where $\mathcal{Z}_{1,2\delta}$ denotes the half-space in $\R^{n+1}$ that contains $\mathcal{H}_1(2\delta)$, $\p \mathcal{Z}_{1,2\delta}$ has unit inward pointing normal $w_1$ and $\p\mathcal{H}_1 \subset\p \mathcal{Z}_{1,2\delta}$.

Define the set \(\displaystyle{\mathcal{A}:=\{s\in[0,s_0]\colon (M+s\nu+t_{0}w_{1})\cap M=\varnothing\}},\) and let $s_{1}:=\inf\mathcal{A}>0.$ Notice that from our assumptions on $s_0$ and $\epsilon$ we have $s_1>0$. We have two possibilities for $s_1$: either $s_{1}\notin\mathcal{A}$ or $s_{1}\in\mathcal{A}$. The first case implies that $M+s_{1}\nu+t_{0}w_{1}$ and $M$ have points of contact, which is impossible by the maximum principle and our hypothesis over $M.$ Consequently it holds $s_{1}\in\mathcal{A},$ and so \(\displaystyle{\dist\left(M+s_{1}\nu+t_{0}w_{1},M\right)}=0\) and $\textstyle{\{M+s_{1}\nu+t_{0}w_{1}\}\cap M=\varnothing}.$ This fact together with our choice of $\epsilon$ imply that there exist sequences $\{p_{i}\}$ in $M\setminus \mathcal{C}(s)$ and $\{q_{i}\}$ in $(M\setminus\mathcal{C}(s))+s_{1}\nu+t_{0}w_{1}$ such that 
$\dist(p_i,\mathcal{C}(s)\cap M)>2\epsilon,$ $\dist(q_i,\mathcal{C}(s)\cap M)>2\epsilon,$ $\dist(p_i,\mathcal{C}(s)\cap M+s_{1}\nu+t_{0}w_{1})>2\epsilon,$ $\dist(q_i,\mathcal{C}(s)\cap M+s_{1}\nu+t_{0}w_{1})>2\epsilon$ and
$\dist(p_{i},q_{i})=0.$ Observe that we can assume $\{\langle q_{i},\ee_{1}\rangle\}, \{\langle p_{i},\ee_{1}\rangle\}\to a$ and $\{\langle q_{i},\ee_{n}\rangle\}, \{\langle p_{i},\ee_{n}\rangle\}\to b$. 

In $\R^{n+1}\setminus(\mathcal{C}(s)\cup\mathcal{C}(s)+s_{1}\nu+t_{0}w_{1})$ consider the following sequences \[\displaystyle{\{M_{i}:=(M^1\setminus\mathcal{C}(s))-(0,p_{2},\ldots,p_{n-1},0,p_{n+1})\}}\] and \[\displaystyle{\{\widehat{M}_{i}:=(M^2\setminus\mathcal{C}(s))+s_{1}\nu+t_{0}w_{1}-(0,q_{2},\ldots,q_{n-1},0,q_{n+1})\}},\] where $M_i$ indicates the wing of $M$ which is asymptotic to $\mathcal{H}_i$. In particular, each $M_i$ and $\widehat{M}_i$ are graphs over open half-hyperplanes in $\mathcal{H}_1$. Hence, they are stable hypersurfaces and the sequences $\{M_i\}$ and $\{\widehat{M}_i\}$ have locally bounded area by Proposition \ref{homology-general.}. 

Now \cite{Wickramasekera}[Theorem 18.1] implies that, after passing to a subsequence, $M_{i}\rightharpoonup M_{\infty}$ and $\widehat{M}_{i}\rightharpoonup\widehat{M}_{\infty}$ in $\R^{n+1}\setminus(\mathcal{C}(s)\cup\mathcal{C}(s)+s_{1}\nu+t_{0}w_{1}),$ where $M_{\infty}$ and $\widehat{M}_{\infty}$ are connected stable integral varifold so that ${\rm sing}\ M_\infty$ and ${\rm sing}\ \widehat{M}_\infty$ satisfy $\mathcal{H}^{n-7}({\rm sing}\ M_\infty)=\mathcal{H}^{n-7}({\rm sing}\ \widehat{M}_\infty)=0$ in $\R^{n+1}\setminus(\mathcal{C}(s)\cup\mathcal{C}(s)+s_{1}\nu+t_{0}w_{1})$, and $(a,0,\ldots,0,b,0)\in\spt M_{\infty}\cap\spt \widehat{M}_{\infty}$ by \cite{Simon}[Corollary 17.8]. The connectedness of the support can be obtained as in \cite{Gama-Martin}[Lemma 3.1]. 

On the other hand, \cite{Ilmanen-maximum}[Theorem A (ii)] implies that ${\rm reg}\ M_{\infty}$ and ${\rm reg}\ \widehat{M}_{\infty}$ are connected subsets in $\R^{n+1}\setminus(\mathcal{C}(s)\cup\mathcal{C}(s)+s_{1}\nu+t_{0}w_{1})$. Consequently, the asymptotic behaviour of $\spt M_{\infty}$ and $\spt \widehat{M}_{\infty}$ imply that ${\rm reg}\ M_{\infty}$ does not intersect ${\rm reg}\ \widehat{M}_{\infty}$. Thus, one has $\spt M_{\infty}\cap\spt \widehat{M}_{\infty}\subset{\rm sing}\ M_{\infty}\cup{\rm sing}\ \widehat{M}_{\infty},$ and so, we would have $\mathcal{H}^{n-1}(\spt M_{\infty}\cap\spt \widehat{M}_{\infty})=0,$ so \cite{Wickramasekera}[Theorem 19.1 (a)] implies that $\spt M_{\infty}\cap\spt \widehat{M}_{\infty}=\varnothing$, which is impossible since $(a,0,\ldots,0,b,0)\in\spt M_{\infty}\cap\spt \widehat{M}_{\infty}.$ Therefore, we must have $\Pi_1=\Pi_2$. However, if we proceed as in Claim \ref{slablimitation} we conclude $M=\Pi_1$. This concludes the proof of the theorem. 
\end{proof}

\bibliographystyle{amsplain, amsalpha}

\end{document}